\documentstyle[amssymb,amsfonts]{amsart}

\def\1{{\hbox{\bf 1}}}

\def\R{{\hbox{\bf R}}}
\def\I{{\hbox{\bf I}}}

\def\F{{\hbox{\bf F}}}
\def\K{{\hbox{\bf K}}}
\def\Trace{{\hbox{\bf Trace}}}

\font \roman = cmr10 at 10 true pt

\def\be#1{ \begin{equation}\label{#1} }
\def\bas{\begin{align*}}
\def\eas{\end{align*}}
\def\bi{\begin{itemize}}
\def\ei{\end{itemize}}

\def\dist{{\hbox{\roman dist}}}

\def\Z{{\hbox{\bf Z}}}

\newenvironment{proof}{\noindent {\bf Proof} }{\endprf\par}
\def \endprf{\hfill  {\vrule height6pt width6pt depth0pt}\medskip}
\def\emph#1{{\it #1}}
\def\textbf#1{{\bf #1}}

\def\Root{\rm {Root}}

\theoremstyle{plain}
  \newtheorem{theorem}[subsection]{Theorem}

  \newtheorem{lemma}[subsection]{Lemma}
  \newtheorem{corollary}[subsection]{Corollary}
  \newtheorem{question}[subsection]{Question}
\theoremstyle{remark}
  \newtheorem{remark}[subsection]{Remark}

\theoremstyle{definition}
  \newtheorem{definition}[subsection]{Definition}
\include{psfig}
\begin{document}
\title[Sum-product estimates  via directed expanders]
{ Sum-product  estimates  via directed expanders}

\author{Van Vu}
\address{Department of Mathematics, Rutgers, Piscataway, CA 08854}
\email{vanvu@@math.rutgers.edu}
\thanks{The author is supported by an NSF Career Grant.}

\begin{abstract} Let $\F_q$ be a finite field of order $q$
and $P$ be a polynomial in $\F_q[x_1, x_2]$. For a set $A \subset
\F_q$, define $P(A):=\{P(x_1, x_2) | x_i \in A \}$.  Using certain
constructions of expanders, we characterize all polynomials $P$ for
which the following holds

\vskip2mm \centerline{\it If $|A+A|$ is small, then $|P(A)|$ is
large.} \vskip2mm

\noindent The case $P=x_1x_2$ corresponds to the well-known
sum-product problem.
\end{abstract}

\maketitle

\section{Introduction}

Let $Z$ be a ring and $A$ be a finite subset of $Z$. The {\it
sum-product} phenomenon, first investigated in \cite{ESz}, can be
expressed as follows

\vskip2mm

\centerline{\it If $|A+A|$ is small, then $|A \cdot A|$ is large.
$(\ast)$}

The earlier works focused on the case $Z$ is $\R$ or $\Z$. In the
last few years, starting with the paper \cite{BKT}, the case when
$Z$ is a finite field or a modular ring has been studied
extensively, leading to many important contributions in various
areas of mathematics (see \cite{Bour2} for a partial survey).

One of the main applications of sum-product estimates is new
constructions of expanders (see, e.g., \cite{BG}). In this paper,
we investigate the reversed direction and derive sum-product
estimates  from certain constructions of expanders. In fact, our
arguments lead to more general results, described below.

Let $\F_q$ be a finite field and $P$ be a polynomial in $\F_q[x_1,
x_2]$. For a set $A \subset \F_q$, define $P(A):=\{P(x_1, x_2) | x_i
\in A \}$. As a generalization of $(\ast)$ (which is the case
$P=x_1x_2$), it is tempting  to consider the following statement

\vskip2mm

\centerline{\it If $|A+A|$ is small, then $|P(A)|$ is large.
$(\ast')$}

\vskip2mm

A short consideration reveals, however, that  $(\ast')$ does not
hold for some classes of polynomials. For instance, if $P$ is linear
and $A$ is an arithmetic progression, then both $|A+A|$ and $|P(A)|$
are small.

{\bf Example.} Set $P_1:=2x_1 + 3x_2$. Let $A=\{1, \dots, n \}
\subset \F_q$, where $q$ is a prime and  $1 \le n \ll q$. Then
$|A+A| =2n-1$ and $|P_1(A)| = 5n-4$.

More generally, if $A$ is an arithmetic progression or  a
generalized arithmetic progression and $Q$ is a polynomial in one
variable and $L$ is a linear form, then both $|A+A|$ and $|P(A)|$
can be  small   for $P:=Q(L(x_1,x_2))$.

{\bf Example.} Set $P_2:=(2x_1 + 3x_2)^2 - 5(2x_1 + 3x_2) +3$. Let
$A=\{1, \dots, n \} \subset \F_q$, where $q$ is a prime and  $1
\le n \ll q$. Then $|A+A| =2n-1$ and $|P_2(A)| =  |P_1(A)| =
5n-4$. In this case, $Q= z^2 -5z+3$ and $L= P_1= 2x_1 +3x_2$.

Our main result shows that $P:=Q(L(x_1,x_2))$ is the {\it only}
(bad) case where the more general phenomenon $(\ast)'$ fails.

\vskip2mm

 \begin{definition} \label{def:deg} A polynomial $P \in \F_q[x_1,x_2]$ is {\it degenerate}
 if it is  of the form
$Q(L(x_1, x_2))$ where $Q$ is an one-variable polynomial and $L$ is
a linear form in $x_1,x_2$. \end{definition}

The following refinement of $(\ast')$ holds

\vskip2mm

\centerline{\it If $|A+A|$ is small and $P$ is non-degenerate,
then $|P(A)|$ is large. $(\ast \ast)$}

\vskip2mm

\begin{theorem} \label{theorem:1} There
is a positive constant $\delta$ such that the following holds. Let
$P$  be a non-degenerate polynomial of degree $k$ in
$\F_q(x_1,x_2]$. Then for any $A \subset \F_q$

$$\max \{|A+A|, |P(A)| \} \ge  |A| \min \{
\delta (\frac{|A|^2}{k^4q})^{1/4}, \delta ( \frac{q}{k|A|})^{1/3} \}
.
$$
\end{theorem}

\begin{remark} \label{remark:2} The estimate in Theorem
\ref{theorem:1} is non-trivial when $k^2q^{1/2} \ll |A| \ll q/k$. In
the case when $P$ has fixed degree, this means  $q^{1/2} \ll |A| \ll
q$. This assumption is necessary as if $A$ is a subfield of size $q$
or $q^{1/2}$ then $|A+A|=|A|$ and $|P(A)|$ is at most $|A|$.
\end{remark}

\begin{remark} \label{remark:1} Since  $P=x_1x_2$ is clearly non-degenerate, we
obtain the following sum-product estimate, reproving a result from
\cite{Iosevich}

$$\max \{|A+A|, |A \cdot A| \} \ge  |A| \min \{
 \delta (\frac{|A|^2}{q})^{1/4}, \delta (\frac{q}{|A|})^{1/3} \}.
$$

\end{remark}

 Our arguments can be extended to  modular rings. Let $m$ be a large integer and
  $\Z_m$ be the ring modular $m$. Let $\gamma(m)$ be the smallest prime
divisor of $m$ and $\tau(m)$ be the number of divisors of $m$.
Define $g(m):= \sum_{ n | m} \tau (m) \tau (m/n)$.

\begin{theorem}\label{theorem:2} There is a positive constant $\delta$ such that the following holds.
 Let $A$ be a subset of $\Z_m$. Then
$$\max \{|A+A|, |A \cdot A| \}  \ge  |A| \min \{
\delta \frac{\gamma (m)^{1/4} |A|^{1/2}} {m^{1/2} }, \delta
(\frac{m}{|A|})^{1/3} \}.
$$

\end{theorem}

\begin{remark} \label{remark:3} This theorem is effective when
$m$ is the product of few large primes.
\end{remark}

Our study was motivated by two papers \cite{Terras} and
\cite{Iosevich}. In these papers, the authors used an argument based
on  Kloosterman sums estimates to study Cayley graphs and the
sum-product problem, respectively. Our approach here relies on a
combination of a generalization of this argument and the spectral
method from graph theory.

\section {Erd\"os' distinct distances problem}

The following question, asked by Erd\"os in the 1940's \cite{Erdos},
is among the most well known problems in discrete geometry

\begin{question} \label{question:distance}
What is the minimum number of distinct distances (in euclidean
norm)  determined by $n$ points on the plane ?
\end{question}

For a point set $A$, we denote by $\Delta(A)$ the set of distinct
distances in $A$. It is easy to show that $|\Delta (A)| = \Omega
(|A|^{1/2})$. To see this, consider an arbitrary point $a \in A$.
If from $a$ there are $|A|^{1/2}$ different distances, then we are
done. Otherwise, by the pigeon hole principle, there is a circle
centered at $a$ containing at least $|A|^{1/2}$ other points. Take
a point $a' $ on this circle. Since two circles intersect in at
most $2$ points, there are at least $\frac{|A|^{1/2} -1}{2}$
distinct distances from $a'$ to the other points on the circle.

It has been conjectured that $|\Delta (A)| \ge |A|^{1-o(1)}$ (the
$o(1)$ term is necessary as shown by the square grid). This
conjecture is still open. For the state of art of this problem, we
refer to \cite[Chapter 6]{TVbook}.

What happens if one replaces the euclidean distance by  other
distances ? One can easily see that for the $l_1$ distance, the
conjectured bound
 $|\Delta (A)| \ge |A|^{1-o(1)}$ fails, as the square grid determines only $|A|^{1/2}$ distances.
 On the other hand, it seems reasonable to think that there is no
 essential difference between the $l_2$ and (say) the $l_4$ norms. In
 fact, in \cite{Taostudent}, it was shown that certain arguments used to
 handle the $l_2$ case can be used, with some more care,
 to handle a wide class of other distances.

The finite field version of Erd\"os problem was first considered
in \cite{BKT}, with the euclidean distance. Here we extend this
work for a general distance. Let $P$ be a symmetric polynomial in
two variables. (By symmetry, we mean that $P$ is symmetric around
the origin, i.e., $P(x,y)=P(-x,-y)$.)  Define the $P$-distance
between two points $x=(x_1,x_2)$ and $y=(y_1,y_2)$ in the finite
plane $\F_q^2$ as $P(y_1-x_1, y_2-x_2)$. Let $\Delta_P(A)$ be the
set of distinct $P$-distances in $A$.

\begin{theorem} \label{theorem:P-distances}
There is a positive constant $\delta$ such that the following holds.
Let $P$ be a symmetric non-degenerate polynomial of degree $k$ and
$A$ be a subset of the finite plane $\F_q^2$, then

$$|\Delta_P (A)| \ge \delta \min \{ \frac{|A|} {k^2\sqrt q}, \frac{q}{k} \}. $$

\end{theorem}

\begin{remark}
The polynomial $P= x^p + y^p$, which corresponds to the $l_p$
norm, is non-degenerate for any positive integer $p \ge 2$.
\end{remark}

\begin{remark} Assume that $k=O(1)$.
For $|A| \gg q$, the term $ \frac{|A|} {\sqrt q} \gg |A|^{1/2}$, and
so $|\Delta_P(A)| \gg |A|^{1/2}$. If $A| \le q$, one cannot expect a
bound better than $|A|^{1/2}$, as $A$ can be a sub-plane.
\end{remark}

\begin{remark}
The proof also works for a non-symmetric $P$. In this case,
$\dist(x,y)$ and $\dist(y,x)$ may be different.
\end{remark}

\section {Directed expanders and spectral gaps}

Let $G$ be a $d$-regular graph on $n$ vertices and $A_G$ be the
adjacency matrix of $G$. The rows and columns of $A_G$ are indexed
by the vertices of $G$ and the entry $a_{ij}=1$ if $i$ is adjacent
to $j$ in $G$ and zero otherwise. Let $d=\lambda_1 (G) \ge \lambda_2
(G) \ge \dots \ge \lambda_n (G)$ be the eigenvalues of $A_G$. Define

$$\lambda(G):= \max \{|\lambda_2|, |\lambda_n| \}. $$

It is well known that if $\lambda(G)$ is significantly less than
$d$, then $G$  behaves like a random graphs (see, for example,
\cite{CGW} or \cite{AS}). In particular, for any two vertex sets
$B$ and $C$

$$|e(B,C) - \frac{d}{n} |B||C| | \le \lambda(G) \sqrt {|B||C|}. $$

\noindent where $e(B,C)$ is the number of edges with one end point
in $B$ and the other in $C$.

We are going to develop a directed version of this statement. Let
$G$ be a directed graph (digraph) on $n$ points where the out-degree
of each vertex is $d$. The adjacency matrix $A_G$ is defined as
follows: $a_{ij}=1$ if there is a directed edge from $i$ to $j$ and
zero otherwise. Let $d= \lambda_1 (G), \lambda_2 (G), \dots,
\lambda_n (G)$ be the eigenvalues of $A_G$. (These numbers can be
complex so we cannot order them, but by Frobenius' theorem all
$|\lambda_i| \le d$.) Define

$$\lambda(G):= \max_{i \ge 2} |\lambda_i|. $$

An $n$ by $n$  matrix $A$ is {\it normal} if $A^{\top} A= A
A^{\top}$. We say that a digraph is normal if its adjacency matrix
is a normal matrix. There is a simple way to test whether a
digraph is normal. In a digraph $G$, let $N^{+}(x,y)$ be the set
of vertices $z$ such that both $xz$ and $yz$ are (directed) edges.
Similarly, let $N^{-}(x,y)$ be the set of vertices $z$ such that
both $zx$ and $zy$ are (directed) edges. It is easy to see that
$G$ is normal if and only if

\begin{equation} \label{eqn:normal} |N^{+}(x,y) | = | N^{-}(x,y) |
\end{equation}

\noindent for any two vertices $x$ and $y$.

\begin{lemma} \label{lemma:expanding}  Let $G$ be a normal directed graph on $n$ vertices with all
out-degree equal $d$.  Let $d= \lambda_1 (G), \lambda_2 (G),
\dots, \lambda_n (G)$ be the eigenvalues of $A_G$. Then for any
two vertex sets $B$ and $C$

$$\Big| e(B,C) - \frac{d}{n} |B||C| \Big| \le \lambda(G) \sqrt {|B||C|}. $$

\noindent where $e(B,C)$ is the number of (directed) edges from
$B$ to $C$.
\end{lemma}

\begin{proof} The eigenvector of $\lambda_1=d$ is $\1$, the
all-one vector. Let $v_i$, $2 \le i \le n$, be the eigenvectors of
$\lambda_i$. A well known fact from linear algebra asserts  that
if $A$ is normal then its eigenvectors  form an orthogonal bases
of $\K^n$. It follows that   any vector $x$ orthogonal to $\1$ can
be written as a linear combination of these $v_i$. By the
definition of $\lambda$ we have that for any such vector $x$

$$\|A_g x \|^2 = <A_g x, A_G x> \le \lambda^2 \|x\|^2. $$

From here one can use  the same arguments as in the non-directed
case (following \cite{AS}) to conclude the proof. We reproduce
these arguments for the reader's convenience.

Let $V:= \{1, \dots, n \}$ be the vertex set of $G$. Set $c:=
|C|/n$ and let $x := (x_1, \dots, x_n)$ where $x_i := \I_{i \in C}
- b$. It is clear that $x$ is orthogonal to $\1$. Thus,

$$ < Ax, Ax > \le \lambda(G)^2 \|x \|^2 . $$

The right hand side is $\lambda_G^2 c(1-c) n \le \lambda_G^2 cn
=\lambda(G)^2 |C|$. The left hand side is $\sum_{v \in V}
(|N_C(v)| -cd )^2 $, where $N_C(v)$ is the set of $v' \in C$ such
that $vv'$ is an directed edge. It follows that

\begin{equation}\label{eqn:eigen1} \sum_{v \in B} (|N_C(v) - cd)^2 \le \sum_{v \in V} (|N_C(v)|
-cd )^2   \le \lambda^2 |C|. \end{equation}


 On the other hand, by the triangle inequality

\begin{equation} \label{eqn:eigen2} |e(B,C) - \frac{d}{n}|B||C| | = |e(B,C) - cd|B| | \le \sum_{v
\in B} |N_C(v)- cd |. \end{equation}

\noindent By Cauchy-Schwartz and \eqref{eqn:eigen1}, the right
hand side of \eqref{eqn:eigen2} is bounded from above by

$$\sqrt{|B|} (\sum_{ v \in B} (N_C(v)- cd)^2)^{1/2} \le \lambda
\sqrt{|B||C|},
$$

\noindent concluding the proof.
\end{proof}

Now we are ready to formalize our first main lemma:

\begin{lemma} (Expander decomposition lemma)
Let $\overrightarrow{K_n}$ be the complete digraph on $V:=\{1,
\dots, n \}$. Assume that $\overrightarrow{K_n}$  is decomposed in
to $k+1$ edge-disjoint digraphs $H_0, H_1, \dots, H_k$ such that

\begin{itemize}

\item For each  $i=1 , \dots, k $, the out-degrees in $H_i$ are
the same and at most $d$ and $\lambda(H_i) \le \lambda$.

\item The out-degrees in $H_0$ are at most $d'$.

\end{itemize}

Let $B$ and $C$ be subsets of $V$ and $K$ be a subgraph of
$\overrightarrow{K_n}$  with $L$ (directed) edges going from $B$
to $C$. Then $K$ contains edges from at least

$$\min \{ \frac{L- |B|d'}{2 \lambda \sqrt{|B||C|} }, \frac{(L- |B|d') n}{2d
|B||C|}  \}$$

\noindent different $H_i$, $i \ge 1$.
\end{lemma}

\begin{proof} By the previous lemma, each $H_i$,
$1 \le i \le k$ has at most

$$ \frac{d}{n} |B||C|  + \lambda \sqrt{|B||C|} $$

\noindent edges going from $B$ to $C$. Furthermore, $H_0$ has at
most $|B|d'$ edges going from $B$ to $C$. Thus the number of
$H_i$, $i \ge 1$, having edges in $K$ is at least

$$ (\frac{d}{n} |B||C| + \lambda \sqrt {|B||C|})^{-1} ( L -  d'|B| )  \ge
\min \{ \frac{L- d'|B|}{2 \lambda \sqrt{|B||C|}}, \frac{(L- d'
|B|) n}{2d |B||C|} \}$$

\noindent completing the proof. \end {proof}

\section{Directed Cayley graphs}

Let $H$ be a finite (additive) abelian group and $ S $ be a
 subset of $H$. Define a directed
graph $G_S$ as follows. The vertex set of $G$ is $H$. There is a
direct edge from $x$ to $y$ if and only if  $y-x \in S$. It is
clear that every vertex in $G_S$ has out-degree $|S|$. (In general
$H$ can be non-abelian, but in this paper we restrict ourselves to
this case.)

Let $\chi_{\xi}$, $\xi \in H$, be the (additive) characters of $H$.
It is well known that for any $\xi \in H$, $\sum_{ s\in S}
\chi_{\xi}(s) $ is an eigenvalue of $G_S$, with respect the
eigenvector $(\chi_{\xi} (x))_{x\in H}$.

It is important to notice that the graph  $G_S$, for any $S$, is
normal, using \eqref{eqn:normal}. Indeed, for any two vertex $x$
and $y$

$$|N^{+} (x,y)| = |N^{-} (x,y)| = |(x+S) \cap (y +S)|. $$

We are going to focus on the following two cases

{\it Special case 1.} $H= \F_q^2$, with $\F_q$ being a finite
field of $q= p^r$ elements, $p$ prime.  Using $e(\alpha)$ to
denote $\exp (\frac{2\pi i}{p} \alpha)$, we have

$$\chi_{\xi} (x) = \exp (\frac{2\pi i}{p} \Trace \,\,\xi \cdot x ) =
e(\Trace \,\, \xi \cdot x), $$

\noindent where $\Trace \,\,z := z + z^p + \dots + z^{p^{r-1}}$
and
 $\xi \cdot x$ is the inner product of $\xi$ and $x$.

{\it Special case 2.}  $H= \Z_m^2$. In this case we  use
$e(\alpha)$ to denote $\exp (\frac{2\pi i}{m} \alpha)$. We have

$$ \chi_{\xi} (x) = \exp (\frac{2\pi i}{m} \xi \cdot x ) = e (\xi \cdot x ). $$

Our second main ingredient is  the following theorem, which is a
corollary of  \cite[Theorem 5.1.1]{Katz}. (We would like to thank B.
C. Ngo for pointing out this reference.)

\begin{theorem} \label{theorem:exponential2}  Let
$P$ be a polynomial of degree $k$ in $\F_q [x_1, x_2]$ which does
not contain a linear factor. Let $\Root (P)$ be the set of roots of
$P$ in $\F_q^2$. Then for any $0 \neq y \in \F_q^2$,

$$| \sum_{x \in \Root (P)} e(x \cdot y)| = O(k^2 q^{1/2 }). $$

\end{theorem}

Given a polynomial $P$ and an element $a \in \F_q$, we denote by
$G_a$ the Cayley graph defined by the set $\Root (P-a)$. As a
corollary of the theorem above, we have

\begin{corollary} \label{cor:exponential2}   Let
$P$ be a polynomial of degree $k$ in $\F_q [x_1, x_2]$ and $a$ be an
element of $F_q$ such that $P-a$ does not contain a linear factor.
Then $\lambda (G_a) = O(k^2q^{1/2})$.
\end{corollary}

It is plausible that  a ring analogue of Theorem
\ref{theorem:exponential2} can be derived (with $\F_q$ replaced by
$\Z_m$). However, the (algebraic) machinery involved is heavy.  We
shall give a direct proof for Corollary \ref{cor:exponential2} in
the special case when  $P$ is quadratic.

Let $\Omega$ be the set of those quadratic polynomials  which (after
a proper changing of variables) can be written in the form $A_1x^2 +
A_2y^2$ with $A_1,A_2 \in \Z^{\ast}_m$, the set of elements co-prime
with $m$. (For example, both $Q=x^2+y^2$ and $Q= 2xy = (x+y)^2
-(x-y)^2$ belong to $\Omega$.) Fix a $Q$ in $\Omega $ and for each
$a \in \Z_m$ define the Cayley graph $G_a$ as before.

\begin{theorem} \label{theorem:gap2}
For any  $0 \neq a \in \Z_m$,

$$\lambda(G_a) \le  g(m) \frac{m}{ \gamma (m)^{1/2}}. $$
\end{theorem}

The proof of this theorem will appear in Section \ref{section:last}.


\section {Proofs of Theorems \ref{theorem:1} and \ref{theorem:2}}

To prove Theorem \ref{theorem:1}, consider a set $A \subset \F_q$
and set $B := A \oplus  A \subset \F_q^2$. Since our estimate is
trivial if $|A| =O(k^2 q^{1/2})$, we assume that $|A| \gg
k^2q^{1/2}$.

 For each $a \in \F_q$, consider the polynomial $P_a= P-a$
and define a Cayley graph $G_a$ accordingly. The out-degree in
this graph is $O(q)$. We say that an element $a$ is {\it good} if
$P-a$ does not contain a linear factor and $\it bad$ other wise.

\begin{lemma} Let  $P $ be a polynomial of degree $k$ in $\F_q[x_1, \dots, x_d]$. Assume that $P$
cannot be written in the form $P= Q(L)$, where $Q $ a polynomial
with one variable and $L$ is a linear form of $x_1, \dots, x_d$.
Then there are at most $k-1$ elements $a_i$ such that the
polynomial $P-a_i$ contains a linear factor.
\end{lemma}

\begin{proof}  Let $a_1, \dots, a_k$ be different elements of $\F_q$ such
that there are linear forms $L_1, \dots, L_k$ and polynomials
$P_1, \dots, P_k \in \F_q[x_1, \dots, x_d]$ such that $P-a_i= L_i
P_i$.

If $L_i$ and $L_j$ had a common root $x$, then $P(x)-a_i =P(x)-a_j
=0$, a contradiction as $a_i \neq a_j$. It follows that for any $1
\le i < j \le d$, $L-i$ and $L_j$ do not have  a  common root. But
since the $L_i$ are linear forms, we can conclude that they are
translates of the same linear form $L$, i.e., $L_i =L -b_i$, for
some $b_1, \dots, b_k \in \F_q$.

It now suffices to  prove the following claim

\begin{lemma} Let $P$ be a polynomial in $\F_q [x_1, \dots, x_d]$ of degree
$k$. Assume that there is a non-zero linear form $L$, a sequence
$a_1, \dots, a_k$ of (not necessarily distinct) elements of $\F_q$
and a set  $ \{b_1, \dots, b_k\} \subset \F_q$ such that
$P(x)=a_i$ whenever $L (x)=b_i$. Then there is a polynomial $Q$ in
one variable such that $P= Q(L)$.
\end{lemma}

Assume, without loss of generality, that the coefficient of $x_1$
in $L$ is non-zero. We are going to induct on the degree of $x_1$
in $P$ (which is at most $k$). If this degree is $0$ (in other
words $P$ does not depend on $x_1$), then $P$ is a constant, since
for any sequence $x_2, \dots, x_d$, we can choose an $x_1$ such
that $L(x_1, \dots, x_d)=b_1$, so

$$ P(x_1, \dots, x_d)= a_1. $$

If the degree in concern is not zero, then we can write

$$ P = (L-b_1)  P_1' + Q_1 , $$

\noindent where $Q_1$ does not contain $x_1$. By the above
argument, we can show that $Q_1=a_1$. Furthermore, if $L(x)=b_i$,
$2\le i \le k$, then $Q_1= (a_i-a_1)/(b_i-b_1)$. Now apply the
induction hypothesis on $P_1'$, whose $x_1$-degree is one less
than that of $P_1$.
\end{proof}

If $a$ is good, then $\lambda (G_a)= O(k^2q^{1/2})$. Let the graph
$H_0$ be the union of bad $G_a$. By the above lemma, the maximum
out-degree of this graph is $d' =O(k^2q)$.


In $\F_q^2$,  define a directed graph $K$ by drawing a directed edge
from $(x,y)$ to $(x',y')$ if and only if either both $x'-x$ and
$y'-y$ are in $A$ or  both $x-x'$ and $y-y'$ are in $A$. Consider
the set $C := (A+A) \oplus (A+A) \subset \F_q^2$. Notice that in $K$
any point from $B$ has  at least $|A|^{2}$ edges going into $C$.
Thus $L$, the number of directed edges from $B$ to $C$, is at least
$|A|^{4}$. Since $|A| \ge k^2q^{1/2}$, we have

$$L- |B|d' \ge |A|^{4} - |B|d' = |A|^{4} - |A|^{2} O(k^2q) =
(1-o(1))|A|^{4}. $$

Applying the Expander Decomposition Lemma and Corollary
\ref{cor:exponential2}, we can conclude
 that the number of
$P_a$ having edges from $B$ to $C$ (which, by definition of $B$
and $C$, is $|P(A)|)$, is at least

$$  \Omega \Big( \min \{ (1-o(1))\frac{ |A|^{4} }{ k^2q^{1/2} |A| |A+A|},
\frac{(1-o(1))|A|^{4} q^2} {kq |A|^{2} |A+A|^{2} } \} \Big). $$

\noindent from which the desired estimate follows by Holder
inequality. The proof of Theorem \ref{theorem:2} (using Theorem
\ref{theorem:gap2} instead of Corollary \ref{cor:exponential2}) is
similar and is left as an exercise.

To prove Theorem \ref{theorem:P-distances}, consider a set $A
\subset \F_q^2$ where $|A| \gg kq$.  Let $B=C=A$ and $K$ be the
complete digraph on $A$. We can assume that $|A\ \gg q$. We have $L
= (1+o(1))|A|^2$ and $d'= O(kq)$. Thus $L- d'|B| = L - d'|A| =
(1+o(1))L= (1+o(1))|A|^2$. By the Expander Decomposition Lemma,

$$|\Delta_P (A)| =  \Omega \Big( \min \{ (1-o(1))\frac{ |A|^{2} }{ k^2q^{1/2} |A|},
\frac{(1-o(1))|A|^{2} q^2} {kq |A|^{2}  } \} \Big). $$

\noindent The right hand side is

$$ \Omega \Big( \min \{\frac{ |A| }{ k^2q^{1/2} }, \frac{q}{k}  \} \Big), $$
completing the proof.

\section{ Proof of Theorem  \ref{theorem:gap2}}
\label{section:last}

We are going to follow an approach from \cite{Terras}. We need to
use the following two classical estimates (see, for example,
\cite[page 19]{KI})

\begin{theorem} \label{theorem:exponential1} (Gauss sum) Let $m$ be an positive odd
integer. Then for any integer $z$ co-prime to $m$

$$|\sum_{y \in \Z_m } e (zy^2)| = \sqrt m . $$

\end{theorem}

\begin{theorem} \label{theorem:kloosterman} (Kloosterman sum) Let $m$ be an positive odd
integer. Then

$$|\sum_{y \in \Z_m^{\ast} } e (ay + b \bar y)| \le \tau(m) (a,b,m)^{1/2} \sqrt m , $$

\noindent where $(a,b,m)$ is the greatest common divisor of $a,b$
and $m$ and $\bar y $ is the inverse of $y$.
\end{theorem}

 Let $p_1 < \dots < p_k$ be
the prime divisors of $m$ and set $\Omega (m) := \{ \prod_{i \in
I} p_i | I \subset \{1, \dots, k \}, I \neq \emptyset \}$. Notice
that $g(m)$ satisfies the following recursive formula: $g(1) :=0$
and $g(m) := \tau (m) + \sum_{d \in \Omega (m)} g(m/d)$.

 Let $S$ be the set of roots of $Q-a$. We are
going to use the notation $G_S$ instead of $G_a$.

 We use induction on $m$ to show that

$$ |\lambda (G_S)| \le g(m) \frac{m}{\gamma (m)^{1/2} }. $$

The case $m=1$ is trivial, so from now on we assume $m >1$. By
properties of Cayley's graphs, the eigenvalues of $G_S$ are

$$\lambda_{\xi} = \sum_{s \in S} e ( \xi \cdot s), $$

\noindent where $\xi \in \Z_m^2$. For $\xi=0$, we obtain the
largest eigenvalue $|S|$, which is the degree of the graph. In
what follows, we assume that $\xi \neq 0$.
 Recall that $s \in S$ if and
only if $Q(s) =a$. We have

\begin{equation} \label{equ8} m \lambda_{ \xi} = \sum_{x \in \Z_m^2} \sum_{ v \in \Z_m}
e(-av) e(  \xi \cdot x + v Q(x)) = \sum_{ v \in \Z_m \backslash
\{0\}} F(v)
\end{equation}

\noindent where $F(v):= \sum_{x \in \Z_m^2} e(-av) e( \xi \cdot x
+ v Q(x))$, taking into account the fact that $F(0)=0$.

 For $d = \prod_{i \in I} p_i \in \Omega
(m)$, let $\eta(d) = |I|+1$. By the exclusion-inclusion formula,

\begin{equation} \label{equ9} \sum_{ v \in \Z_m \backslash 0} F(v)= \sum_{ v \in
\Z_m^{\ast} } F(v)  + \sum_{d \in \Omega (m)} (-1)^{\eta(d)}
\sum_{d | v} F(v). \end{equation}

Let us first bound $S_0:= \sum_{ v \in \Z_m^{\ast} } F(v) $. We
write $x= (x_1,x_2)$ where $x_1, x_2 \in \Z_m$. As $Q$ is
non-degenerate, by changing variables we can rewrite $ e( \xi
\cdot x + v Q(x))$ as $e (v (A_1 x_1^2 + A_2 x_2^2) + (B_1x_1 +
B_2 x_2 + C))$ where
 $B_1, B_2, C$ may depend on $\xi$, but $A_1, A_2 \in \Z_m^{\ast} $ depends only
on $Q$. We have (thanks to the fact that $v, A_1, A_2, 2,4 $ are
all in $\Z_m^{\ast}$)

\begin{align*} &     v (A_1 x_1^2 + A_2 x_2^2) + (B_1x_1 + B_2 x_2 + C) \\ = &  vA_1
(x_1 + \frac{B_1}{2vA_1})^2 + vA_2 (x_2 + \frac{B_2}{2vA_2})^2 +
(C- \frac{B_1^2}{4 vA_1} - \frac{B_2^2}{4 vA_2}). \end{align*}

It follows that

\begin{align} \label{equ2} S_0=  & \sum_{v \in \Z_m^{\ast} }
e(C) e (-av - ( \frac{B_1^2}{4 A_1} + \frac{B_2^2}{4 A_2}) \bar v)
\\ \notag & \times \sum_{x_1, x_2 \in \Z_m} e( vA_1 (x_1 + \frac{B_1}{2vA_1})^2 +
vA_2 (x_2 + \frac{B_2}{2vA_2})^2 ).
\end{align}

Notice that

\begin{equation} \label{equ3}
\sum_{x_1, x_2 \in \Z_m} e( vA_1 (x_1 + \frac{B_1}{2vA_1})^2 +
vA_2 (x_2 + \frac{B_2}{2vA_2})^2 ) = \sum_{y \in \Z_m} e (vA_1
y^2) \sum_{y \in \Z_m} e(vA_2 y^2).
\end{equation}

Set $b:= \frac{B_1^2}{4 vA_1} + \frac{B_2^2}{4 vA_2}$, we have

\begin{equation} \label{equ4} S_0=  e(C) \Big( \sum_{ y \in \Z_m} e (vA_1 y^2)
\sum_{ y \in \Z_m} e (vA_2 y^2)  \Big)\sum_{v \in \Z_m^{\ast} } e
(-av - b\bar v ).
\end{equation}

By Theorems \ref{theorem:exponential1} and
\ref{theorem:kloosterman} and the fact that $(a,b,m) \le
\frac{m}{\gamma (m)}$ (since $a \neq 0$), we have

\begin{equation} \label{equS0} |S_0| \le  m \tau (m)
(a,b,m)^{1/2} m^{1/2} \le \tau (m) \frac{m^2}{ \gamma(m)^{1/2}}.
\end{equation}

Now we bound the second term in the right hand side of
\eqref{equ9}, using the induction hypothesis. Fix $d \in \Omega
(m)$ and consider

$$S_d:= \sum_{d | v} F(v) = \sum_{ x \in \Z_m^2} e( \xi \cdot x) \sum_{ d |v} e
(v(Q(x)-a)). $$

Write $m= dm_d, v= dv'$, where $m_d:= m/d$ and $v' \in \Z_{m_d}$.
Each vector $x$ in $\Z_m^2$ has a unique decomposition $x= x^{[1]} +
m_d x^{[2]}$ where $x^{[1]} \in \Z_{m_d}^2$ and $x^{[2]} \in
\Z_d^2$. Finally, there is $a' \in \Z_{m_d}$ such that $a \equiv a'$
(mod $m_d)$. Since $Q(x)\equiv Q(x^{[1]})$ (mod $m_d$), we have

$$e( v(Q(x)-a)) = \exp \big(\frac{2\pi i}{m_d} v' (Q(x^{[x_1]} - a') \big).
$$

Therefore,

\begin{equation} \label{equ10} \sum_{ d |v} e
(v(Q(x)-a)) = \sum_{v' \in \Z_{m_d}} \exp \big(\frac{2\pi i}{m_d}
v' (Q(x^{[1]} - a' \big) \end{equation}

\noindent which equals $m_d$ if $Q(x^{[1]}) \equiv a'$  (mod
$m_d$) and zero otherwise. It follows that

$$S_d  = m_d \sum_{ x \in \Z_m^2, Q(x^{[1]}) = a' (\hbox{mod} m_d) }
e (\xi \cdot x ). $$

\noindent Next, we rewrite $e(\xi \cdot x)$ as $\exp(\frac{2\pi
i}{m} (\xi \cdot x^{[1]} + m_d \xi \cdot x^{[2]})) $. This way, we
have

\begin{equation} \label{equ11}
S_d = m_d \sum_{x^{[1]} \in \Z_{m_d}^2, Q(x^{[1]}) = a'} e( \xi
\cdot x^{[1]} ) \sum_{ x^{[2]} \in \Z_d^2} \exp (\frac{ 2\pi i}{d}
\xi \cdot x_2). \end{equation}

\noindent The sum $\sum_{ x^{[2]} \in \Z_d^2} \exp (\frac{ 2\pi
i}{d} \xi \cdot x_2)$ is $d^2$ if both coordinates of $\xi$ are
divisible by $d$ and zero other wise. Set $\xi_d= \xi/d$, we have

$$S_d = m_d d^2 \sum_{ Q(x^{[1]})= a'} \exp (\frac{2\pi
i}{m_d} \xi_d \cdot x^{[1]}). $$

\noindent Notice that $\sum_{ Q(x^{[1]})= a'} \exp (\frac{2\pi
i}{m_d} \xi_d \cdot x^{[1]}$ is a (non-trivial) eigenvalue of a
Cayley's graph defined by $Q$ on $\Z_{m_d}^2$, where $m_d=m/d$.
Thus, by the induction hypothesis,

$$| \sum_{ Q(x^{[1]})= a'} \exp (\frac{2\pi
i}{m_d} \xi_d \cdot x^{[1]} )| \le g(m/d) \frac{m/d}{\gamma
(m/d)^{1/2}} \le g(m/d) \frac{m/d} {\gamma (m)^{1/2}}. $$

This implies

\begin{equation} \label{equSd}
|S_d| \le g(m/d) \frac{m^2} {\gamma (m)^{1/2} }.
\end{equation}

By \eqref{equ9}, \eqref{equS0}, \eqref{equSd} and the triangle
inequality

$$m \lambda_{\xi} \le \frac{m^2} {\gamma (m)^{1/2} }\big(\tau (m) +
\sum_{d \in \Omega (m)} g(m/d) \big) = g(m) \frac{m^2}
{\gamma(m)^{1/2}}
$$

\noindent completing the proof.


\begin{thebibliography}{10}

\bibitem{AS} N. Alon, J. Spencer, The probabilistic method (Second
edition),
 Wiley-Interscience, 2000.

\bibitem{BKT} J. Bourgain, N. Katz, T. Tao, Bourgain,
A sum-product estimate in finite fields, and applications, {\it
Geom. Funct. Anal.} 14 (2004), no. 1, 27--57.

\bibitem{BG} J. Bourgain, A.  Gamburd,  New results on
expanders, {\it  C. R. Math. Acad. Sci. Paris} 342 (2006), no. 10,
717--721.

\bibitem{Bour2} J. Bourgain,
More on the sum-product phenomenon in prime fields and its
applications, {\it  Int. J. Number Theory } 1 (2005), no. 1, 1--32.

\bibitem{CGW} F. Chung, R. Graham, R. Wilson, Quasi-random graphs, {\it Combinatorica}  9
 (1989),  no. 4, 345--362.

\bibitem{Erdos} P. Erd\"os,
 On sets of distances of $n$ points, {\it  Amer. Math. Monthly } 53,  (1946). 248--250.

\bibitem{ESz} P. Erd\"os,  E. Szemer\'edi,
On sums and products of integers, {\it Studies in pure mathematics,}
213--218, Birkhäuser, Basel, 1983.

\bibitem{Katz} N. Katz,
Sommes exponentielles, Asterisque 79, Société Mathématique de
France, Paris, 1980.

\bibitem{KI} H. Iwaniec, E. Kowalski, Analytic number theory,
American Mathematical Society, 2004.


\bibitem{Iosevich} D. Hart, A. Iosevich, J. Solymosi,
 Sum-product estimates in finite fields via Kloosterman sums,
{\it to appear in IMRN}.

\bibitem{Pach} P. Brass, W. Moser, J.  Pach,
Research problems in discrete geometry. Springer, New York, 2005.

\bibitem{Terras} A. Medrano, P. Myers, H. Stark, A. Terras, A.
Finite analogues of Euclidean space, {\it J. Comput. Appl. Math.} 68
(1996), no. 1-2, 221--238.



\bibitem{Taostudent} J. Garibaldi, Erdös Distance Problem for Convex
Metrics, {\it Ph. D. Thesis, UCLA 2004}.

\bibitem{TVbook} T. Tao,  V. Vu, Additive combinatorics, Cambridge University Press, 2006.


\end{thebibliography}
\end{document}